\documentclass[reqno,b5paper]{amsart}
\usepackage{amsmath}
\usepackage{amssymb}
\usepackage{amsfonts}
\usepackage{graphics}
\usepackage{amsthm}
\usepackage{enumerate}
\usepackage[mathscr]{eucal}
\newtheorem{theorem}{Theorem}[section]
\newtheorem{lemma}{Lemma}[section]
\newtheorem{corollary}{Corollary}[section]
\newtheorem*{acknowledgement}{\textnormal{\textbf{Acknowledgement}}}

\begin{document}

\title[A note on trans-Sasakian manifolds]
{A note on trans-Sasakian manifolds}
\author[Sharief Deshmukh \and Mukut Mani Tripathi]
{Sharief Deshmukh* \and Mukut Mani Tripathi**}
\newcommand{\acr}{\newline\indent}

\address{\llap{*\,}Department of Mathematics\acr
College of Science\acr
King Saud University\acr
P.O. Box 2455, Riyadh 11451\acr
SAUDI ARABIA}
\email{shariefd@ksu.edu.sa}

\address{\llap{**\,}Department of Mathematics and DST-CIMS\acr
Faculty of Science\acr
Banaras Hindu University\acr
Varanasi 221005\acr
INDIA}
\email{mmtripathi66@yahoo.com}

\thanks{This work was supported by King Saud University, 
Deanship of Scientific Research, College of Science Research Center}

\begin{abstract}
In this paper, we obtain some sufficient conditions for a
$3$-dimensional compact trans-Sasakian manifold of type $(\alpha ,\beta )$
to be homothetic to a Sasakian manifold. A characterization of a $3$%
-dimensional cosymplectic manifold is also obtained.
\end{abstract}

\maketitle

\section{Introduction}

Let $(M,\varphi ,\xi ,\eta ,g)$ be a $(2n+1)$-dimensional 
almost contact metric manifold (cf. \cite{1}). 
Then the product $\overline{M}=M\times R$ has a
natural almost complex structure $J$ with the product metric $G$ being
Hermitian metric. The geometry of the almost Hermitian manifold $(\overline{M%
},J,G)$ dictates the geometry of the almost contact metric manifold $%
(M,\varphi ,\xi ,\eta ,g)$ and gives different structures on $M$ like
Sasakian structure, quasi-Sasakian structure, Kenmotsu structure and others
(cf. \cite{1}, \cite{2}, \cite{7}). 
It is known that there are sixteen different types of
structures on the almost Hermitian manifold $(\overline{M},J,G)$ (cf. \cite{5})
and using the structure in the class $\mathcal{W}_{4}$ \ on $(\overline{M}%
,J,G)$, a structure $(\varphi ,\xi ,\eta ,g,\alpha ,\beta )$ on $M$ called
trans-Sasakian structure, was introduced (cf. \cite{12}) that generalizes
Sasakian and Kenmotsu structures on a contact metric manifold (cf. \cite{2},
\cite{7}), where $\alpha ,\beta $ are smooth functions defined on $M$. Since the
introduction of trans-Sasakian manifolds, very important contributions of
Blair and Oubi\~{n}a \cite{2} and Marrero \cite{10} have appeared, studying the
geometry of trans-Sasakian manifolds. In general a trans-Sasakian manifold $%
(M,\varphi ,\xi ,\eta ,g,\alpha ,\beta )$ is called a trans-Sasakian
manifold of type $(\alpha ,\beta )$. Trans-Sasakian manifolds of type $(0,0)$%
, $(\alpha ,0)$ and $(0,\beta )$ are called cosymplectic, $\alpha $%
-Sasakian, and $\beta $-Kenmotsu manifolds respectively. Marrero \cite{10} has
shown that a trans-Sasakian manifold of dimension $\geq 5$ is either
cosymplectic, or $\alpha $-Sasakian, or $\beta $-Kenmotsu. Since then, there
is a concentration on studying geometry of $3$-dimensional trans-Sasakian
manifolds only (cf. \cite{0}, \cite{3}, \cite{4}, \cite{8}, \cite{9}),
putting some restrictions on the
smooth functions $\alpha ,\beta $ appearing in the definition of
trans-Sasakian manifolds. There are several examples of trans-Sasakian
manifolds constructed mostly on $3$-dimensional Riemannian manifolds (cf.
\cite{2}, \cite{10}, \cite{12}).
Moreover, as the geometry of Sasakian manifolds is very
rich, and is derived from contact geometry, the question of finding
conditions under which a $3$-dimensional trans-Sasakian manifold is
homothetic to a Sasakian manifold becomes more interesting. In this paper we
consider this question and obtain two different sufficient conditions for a
trans-Sasakian manifold to be homothetic to a Sasakian manifold. One of them
is expressed in terms of the smooth functions $\alpha $, $\beta $ and a
bound on certain Ricci curvature, and the other requires that the Reeb
vector should be an eigenvector of the Ricci operator 
(cf. Theorems~\ref{th-3-1},~\ref{th-3-2}).
We also find a characterization of cosymplectic manifolds (cf.
Theorem~\ref{th-4-1}).

\begin{acknowledgement} 
The authors wish to express their sincere thanks to the referee for
many corrections. The first author also wishes to thank the DST-CIMS at
Banaras Hindu University, Varanasi for the hospitality during his visit to
the center.
\end{acknowledgement} 

\section{Preliminaries}

Let $(M,\varphi ,\xi ,\eta ,g)$ be a $3$-dimensional contact metric
manifold, where $\varphi $ is a $(1,1)$-tensor field, $\xi $ a unit vector
field and $\eta $ a smooth $1$-form dual to $\xi $ with respect to the
Riemannian metric $g$ satisfying%
\begin{equation}
\varphi ^{2}=-I+\eta \otimes \xi \text{, }\varphi (\xi )=0\text{, }\eta
\circ \varphi =0\text{, }g(\varphi X,\varphi Y)=g(X,Y)-\eta (X)\eta (Y)\text{,}  \tag{2.1}
\end{equation}%
$X,Y\in \mathfrak{X}(M)$, where $\mathfrak{X}(M)$ is the Lie algebra of
smooth vector fields on $M$ (cf. \cite{1}). If there are smooth functions $\alpha
,\beta $ on an almost contact metric manifold $(M,\varphi ,\xi ,\eta ,g)$
satisfying%
\begin{equation}
(\nabla \varphi )(X,Y)=\alpha \left( g(X,Y)\xi -\eta (Y)X\right) +\beta
\left( g(\varphi X,Y)\xi -\eta (Y)\varphi X\right)\text{,}  \tag{2.2}
\end{equation}%
then this is said to be a trans-Sasakian manifold, where $(\nabla \varphi
)(X,Y)=\nabla _{X}\varphi Y-\varphi (\nabla _{X}Y)$,\quad $X,Y\in \mathfrak{X%
}(M)$ and $\nabla $ is the Levi-Civita connection with respect to the metric
$g$ (cf. \cite{2}, \cite{10}, \cite{12}). We shall denote this trans-Sasakian manifold by $%
(M,\varphi ,\xi ,\eta ,g,\alpha ,\beta )$ and it is called trans-Sasakian
manifold of type $(\alpha ,\beta )$. From equations (2.1) and (2.2), it
follows that%
\begin{equation}
\nabla _{X}\xi =-\alpha \varphi (X)+\beta (X-\eta (X)\xi )\text{,\quad }X\in
\mathfrak{X}(M)\text{.}  \tag{2.3}
\end{equation}%
It is clear that a trans-Sasakian manifold of type $(1,0)$ is a Sasakian
manifold (cf. \cite{1}) and a trans-Sasakian manifold of type $(0,1)$ is a
Kenmotsu manifold (cf. \cite{7}). A trans-Sasakian manifold of type $(0,0)$ is
called a cosymplectic manifold (cf. \cite{6}).

Let $Ric$ be the Ricci tensor of a Riemannian manifold $(M,g)$. Then the
Ricci operator $Q$ is a symmetric tensor field of type $(1,1)$ defined by $%
Ric(X,Y)=g(QX,Y)$, $X,Y\in \mathfrak{X}(M)$. We prepare some tools for
trans-Sasakian manifolds.

\begin{lemma} \label{lem-2-1}
Let $(M,\varphi ,\xi ,\eta ,g,\alpha ,\beta )$ be a $3$%
-dimensional trans-Sasakian manifold. Then $\xi (\alpha )=-2\alpha \beta $.
\end{lemma}

\begin{proof} Using (2.3), we get that 
\[
d\eta (X,Y)=-2\alpha g(\varphi X,Y)
,\qquad X,Y\in \mathfrak{X}(M)
\]
and as a consequence, the $2$-form $\Omega $
defined by $\Omega (X,Y)=\alpha g(\varphi X,Y)$ is closed. Using (2.1),
(2.2), and (2.3) in $d\Omega =0$ after some trivial calculations, we arrive
at
\begin{equation*}
\varphi \left\{ X(\alpha )Y-Y(\alpha )X-2\alpha \beta \eta (Y)X+2\alpha
\beta \eta (X)Y\right\} +g(\varphi X,Y)\left( \nabla \alpha +2\alpha \beta
\xi \right) =0
\end{equation*}%
for all $X,Y\in \mathfrak{X}(M)$. Operating $\varphi $ on the equation above, we get%
\begin{eqnarray*}
&&Y(\alpha )X-X(\alpha )Y+2\alpha \beta \eta (Y)X-2\alpha \beta \eta
(X)Y+X(\alpha )\eta (Y)\xi  \\
&&-Y(\alpha )\eta (X)\xi +g(\varphi X,Y)\varphi \left( \nabla \alpha \right)
\begin{array}{c}
=%
\end{array}%
0\text{.}
\end{eqnarray*}%
For a local orthonormal frame $\left\{ e_{1},e_{2},e_{3}\right\} $ on $M$,
taking $X=e_{i}$ in the equation above, taking the inner product with $e_{i}$
and adding the resulting equations, we get%
\begin{equation*}
\left( 2\alpha \beta +\xi (\alpha )\right) \eta (Y)=0\text{,\quad }Y\in
\mathfrak{X}(M)
\end{equation*}%
which gives%
\begin{equation*}
\left( 2\alpha \beta +\xi (\alpha )\right) \xi =0
\end{equation*}%
and we obtain the result.
\end{proof}

\begin{lemma}
\label{lem-2-2} Let $(M,\varphi ,\xi ,\eta ,g,\alpha ,\beta )$ be a $3$%
-dimensional trans-Sasakian manifold. Then its Ricci operator satisfies%
\begin{equation*}
Q(\xi )=\varphi (\nabla \alpha )-\nabla \beta +2(\alpha ^{2}-\beta ^{2})\xi
-g(\nabla \beta ,\xi )\xi
\end{equation*}%
where $\nabla \alpha ,$ $\nabla \beta $ are gradients of the smooth
functions $\alpha ,$ $\beta $.
\end{lemma}

\begin{proof}
We use (2.1), (2.2), and (2.3) to calculate 
\[
R(X,Y)\xi =\nabla _{X}\nabla _{Y}\xi -\nabla _{Y}\nabla _{X}\xi 
-\nabla _{[X,Y]}\xi 
\]
and after some easy computations we arrive at%
\begin{eqnarray*}
R(X,Y)\xi &=&Y(\alpha )\varphi X-X(\alpha )\varphi Y+X(\beta )(Y-\eta (Y)\xi
)-Y(\beta )(X-\eta (X)\xi ) \\
&&+(\alpha ^{2}-\beta ^{2})(\eta (Y)X-\eta (X)Y)+2\alpha \beta (\eta
(Y)\varphi X-\eta (X)\varphi Y)\text{.}
\end{eqnarray*}%
The above equation gives%
\begin{eqnarray*}
Ric(Y,\xi ) &=&g(\varphi (\nabla \alpha ),Y)-g(\nabla \beta ,Y)-g(\nabla
\beta ,\xi )\eta (Y) \\
&&+2(\alpha ^{2}-\beta ^{2})\eta (Y),
\end{eqnarray*}%
which proves the result.
\end{proof}

Next, we state the following result of \cite{11}, which we shall use in the
sequel.

\begin{theorem} \label{th-2-1} \cite{11} Let $(M,g)$ be a Riemannian manifold.
If $M$ admits a Killing vector field $\xi $ of constant length satisfying
\begin{equation*}
k^{2}\left( \nabla _{X}\nabla _{Y}\xi -\nabla _{\nabla _{X}Y}\xi \right)
=g(Y,\xi )X-g(X,Y)\xi
\end{equation*}%
for a nonzero constant $k$ and any vector fields $X$ and $Y$, then $M$ is
homothetic to a Sasakian manifold.
\end{theorem}

\section{Trans-Sasakian manifolds homothetic to Sasakian manifolds}

In this section we study compact and connected $3$-dimensional
trans-Sasakian manifolds and obtain conditions under which they are
homothetic to Sasakian manifolds. Our first result uses a bound on the Ricci
curvature of the trans-Sasakian manifold in the direction of the vector
field $\xi $.

\begin{theorem} \label{th-3-1}
Let $(M,\varphi ,\xi ,\eta ,g,\alpha ,\beta )$ be a $3$%
-dimensional compact and connected trans-Sasakian manifold. If the Ricci
curvature $Ric(\xi ,\xi )$ satisfies%
\begin{equation*}
0<Ric(\xi ,\xi )\leq 2\left( \alpha ^{2}+\beta ^{2}\right) \text{,}
\end{equation*}%
then $M$ is homothetic to a Sasakian manifold.
\end{theorem}

\begin{proof} Using (2.3) we immediately compute%
\begin{equation}
\delta \eta =\text{div}\xi =2\beta \text{.}  \tag{3.1}
\end{equation}%
Also, since $d\eta (X,Y)=-2\alpha g(\varphi X,Y)$, we obtain%
\begin{equation}
\left\Vert d\eta \right\Vert ^{2}=8\alpha ^{2}\text{.}  \tag{3.2}
\end{equation}%
Now using (2.3), after some obvious calculations, we get%
\begin{equation}
\left\Vert \nabla \xi \right\Vert ^{2}=2(\alpha ^{2}+\beta ^{2})\text{.}
\tag{3.3}
\end{equation}%
Now, using (3.1)-(3.3) in the integral formula (cf. \cite{13})%
\begin{equation*}
\int_{M}\left\{ Ric(\xi ,\xi )-\frac{1}{2}\left\Vert d\eta
\right\Vert ^{2}+\left\Vert \nabla \xi \right\Vert ^{2}-(\delta \eta
)^{2}\right\} =0
\end{equation*}%
and the hypothesis of the theorem, we deduce that%
\begin{equation}
Ric(\xi ,\xi )=2\left( \alpha ^{2}+\beta ^{2}\right) \text{.}  \tag{3.4}
\end{equation}%
Using Lemma~\ref{lem-2-2}, we have%
\begin{equation*}
Ric(\xi ,\xi )=-2\xi (\beta )+2(\alpha ^{2}-\beta ^{2})
\end{equation*}%
which together with (3.4) gives%
\begin{equation}
\xi (\beta )=-2\beta ^{2}\text{.}  \tag{3.5}
\end{equation}

We claim that $\beta $ must be a constant. If $\beta $ is not a constant,
then on the compact $M$ it has a local maximum at some $p\in M$. We have $%
\left( \nabla \beta \right) (p)=0$ and the Hessian $H_{\beta }$ is negative
definite at this point $p$. However, using the equation (3.5), we have $\xi
(\beta )(p)=-2\left( \beta (p)\right) ^{2}=0$ and $H_{\beta }(\xi ,\xi
)(p)=\xi \xi (\beta )(p)=4\left( \beta (p)\right) ^{3}=0$, (where we used $%
\nabla _{\xi }\xi =0$), which yields a contradiction (as the Hessian is
negative definite at $p$). Hence, $\beta $ is a constant and this, combined
with Stokes' theorem applied to $\text{div}(\xi )=2\beta $, proves that $%
\beta =0$.

Since $\beta =0$, the Lemma~\ref{lem-2-1} gives $\xi (\alpha )= 0$.
We claim that $\alpha $ is a constant. If not, on compact $M$ the
smooth function $\alpha $ attains a local maximum at some point $p\in M$.
At this point, the Hessian $H_{\alpha }$ is negative definite.
However, for the unit vector field $\xi $,
we have $H_{\alpha }(\xi ,\xi )=0$, which fails to be negative definite at
point $p$, which is a contradiction. Now, that $\alpha $ is a non-zero
constant follows from the condition in the hypothesis. Thus, using (2.3), we
compute%
\begin{equation*}
\alpha ^{-2}\left( \nabla _{X}\nabla _{Y}\xi -\nabla _{\nabla _{X}Y}\xi
\right) =g(Y,\xi )X-g(X,Y)\xi \text{,}
\end{equation*}%
and this implies by Theorem~\ref{th-2-1} that $M$ is homothetic to a Sasakian
manifold.
\end{proof}

As a direct consequence of the above theorem we have the following result,
which has motivation from the fact that on a $(2n+1)$-dimensional Sasakian
manifold $(M,\varphi ,\xi ,\eta ,g)$ the Ricci operator satisfies $Q(\xi
)=2n\xi $.

\begin{corollary} \label{cor-3-1}
Let $(M,\varphi ,\xi ,\eta ,g,\alpha ,\beta )$ be a $3$%
-dimensional compact and connected trans-Sasakian manifold. If the vector
field $\xi $ satisfies $Q(\xi )=2\alpha ^{2}\xi \neq 0$, then $M$ is
homothetic to a Sasakian manifold.
\end{corollary}

As pointed out earlier, on a $(2n+1)$-dimensional Sasakian manifold $%
(M,\varphi ,\xi ,\eta ,g)$, the Ricci operator satisfies $Q(\xi )=2n\xi $,
that is, the Reeb vector field $\xi $ is an eigenvector of the Ricci
operator. This motivates the question of whether a $3$-dimensional
trans-Sasakian manifold $(M,\varphi ,\xi ,\eta ,g,\alpha ,\beta )$
satisfying $Q(\xi )=\lambda \xi $ for a non-zero constant $\lambda $, is
necessarily homothetic to a Sasakian manifold. We answer this question for
compact connected $3$-dimensional trans-Sasakian manifolds and show that
they are homothetic to Sasakian manifolds.

\begin{theorem} \label{th-3-2}
Let $(M,\varphi ,\xi ,\eta ,g,\alpha ,\beta )$ be a $3$-dimensional
compact and connected trans-Sasakian manifold. Then $M$ is
homothetic to a Sasakian manifold if and only if the vector field $\xi $
satisfies $Q(\xi )=\lambda \xi $ for a non-zero constant $\lambda $.
\end{theorem}

\begin{proof}
Using $Q(\xi )=\lambda \xi $ in Lemma~\ref{lem-2-2}, we have%
\begin{equation}
\varphi (\nabla \alpha )-\nabla \beta =\left( \lambda +\xi (\beta )-2(\alpha
^{2}-\beta ^{2})\right) \xi \text{.}  \tag{3.7}
\end{equation}%
Taking the inner product with $\xi $ in the above equation, we obtain%
\begin{equation}
\xi (\beta )=-\frac{\lambda }{2}+(\alpha ^{2}-\beta ^{2})\text{.}  \tag{3.8}
\end{equation}%
Inserting this value in (3.7), we have%
\begin{equation}
\varphi (\nabla \alpha )-\nabla \beta =\left( \frac{\lambda }{2}-(\alpha
^{2}-\beta ^{2})\right) \xi   \tag{3.9}
\end{equation}%
and applying $\varphi $ to the above equation, we obtain%
\begin{equation}
\nabla \alpha =-2\alpha \beta \xi -\varphi (\nabla \beta )\text{.}
\tag{3.10}
\end{equation}%
If $A$ is a symmetric operator on the trans-Sasakian manifold $M$, we can
choose a local orthonormal frame that diagonalizes $A$ and consequently, we
have%
\begin{equation}
\sum g(\varphi (Ae_{i}),e_{i})=0\text{.}  \tag{3.11}
\end{equation}%
Now for $X\in \mathfrak{X}(M)$, we compute%
\begin{equation*}
\nabla _{X}\left( \varphi (\nabla \beta )+2\alpha \beta \xi \right) =\left(
\nabla _{X}\varphi \right) (\nabla \beta )+\varphi \left( A_{\beta }X\right)
+2X(\alpha \beta )\xi +2\alpha \beta \nabla _{X}\xi
\end{equation*}%
where $A_{\beta }X=\nabla _{X}\nabla \beta $ is a symmetric operator $%
A_{\beta }:\mathfrak{X}(M)\rightarrow \mathfrak{X}(M)$. Taking the inner
product with $X$ in above equation and using equations (2.2) and (2.3),
after some easy calculations we arrive at%
\begin{eqnarray*}
g\left( \nabla _{X}\left( \varphi (\nabla \beta )+2\alpha \beta \xi \right)
,X\right)  &=& \alpha X(\beta )\eta (X)-\alpha \xi (\beta )g(X,X) \\
&& +\; \beta g\left( \varphi X,\nabla \beta \right) \eta (X) \\
&& + \; g\left( \varphi \left( A_{\beta }X\right) ,X\right) +2X(\alpha \beta
)\eta (X) \\
&&+ \; 2\alpha \beta ^{2}g(X,X)-2\alpha \beta ^{2}\left( \eta (X)\right) ^{2}%
\text{.}
\end{eqnarray*}%
Taking trace in the equation above, in view of the equation (3.11), we get%
\begin{equation}
\text{div}\left( \varphi (\nabla \beta )+2\alpha \beta \xi \right) =-2\alpha
\xi (\beta )+2\xi (\alpha \beta )+4\alpha \beta ^{2}=0\text{,}  \tag{3.12}
\end{equation}%
where we used the fact that $\xi (\alpha )=-2\alpha \beta $. Thus using
(3.12) in the equation (3.10), we conclude that $\Delta \alpha =\text{div}%
(\nabla \alpha )=0$ on compact $M$, which proves that $\alpha $ is a
constant. Using the constant $\alpha $ in the equation (3.9), we get%
\begin{equation*}
-\nabla \beta =\left( \frac{\lambda }{2}-(\alpha ^{2}-\beta ^{2})\right) \xi
\end{equation*}%
which together with the equation (3.8) gives%
\begin{eqnarray*}
\Delta \beta  &=&-2\beta \xi (\beta )-\left( \frac{\lambda }{2}-(\alpha
^{2}-\beta ^{2})\right) \text{div}\xi  \\
&=&-2\beta \left( -\frac{\lambda }{2}+(\alpha ^{2}-\beta ^{2})\right)
-2\beta \left( \frac{\lambda }{2}-(\alpha ^{2}-\beta ^{2})\right)  \\
&=&0\text{.}
\end{eqnarray*}%
Here we used the fact that $\text{div}\xi =2\beta $. Thus $\beta $ is a
constant, which together with Stokes' theorem and $\text{div}\xi =2\beta $
proves that $\beta =0$. If $\alpha =0$, then (3.7) would imply $\lambda =0$,
which is a contradiction. Consequently, $\alpha $ is a non-zero constant
which by the equation (3.1) satisfies%
\begin{equation*}
\alpha ^{-2}\left( \nabla _{X}\nabla _{Y}\xi -\nabla _{\nabla _{X}Y}\xi
\right) =g(Y,\xi )-g(X,Y)\xi \text{.}
\end{equation*}%
This proves that $M$ is homothetic to a Sasakian manifold. The converse is
obvious.
\end{proof}

\section{A characterization of cosymplectic manifolds}

In this section, we study $3$-dimensional compact trans-Sasakian manifolds,
and obtain a characterization of cosymplectic manifolds. Let $(M,\varphi
,\xi ,\eta ,g,\alpha ,\beta )$ be a $3$-dimensional trans-Sasakian manifold.
Then for each point $p\in M$ there is a neighbourhood $U$ of $p$, where we
have a local orthonormal frame $\{e,\varphi e,\xi \}$ for a unit vector
field $e$ on $U$ called an adapted frame. Using the equations (2.1), (2.2)
and (2.3), we obtain the following local structure equations defined on $U$%
\begin{equation}
\nabla _{e}\xi = \beta e-\alpha \varphi e, \quad \nabla _{\varphi e}\xi
= \alpha e+\beta \varphi e, \quad \nabla _{\xi }\xi = 0 , \tag{4.1}
\end{equation}
\begin{equation}
\nabla _{e}e = \gamma \varphi e-\beta \xi , \quad
\nabla _{\varphi e}e = -\delta \varphi e -\alpha \xi ,\quad
\nabla _{\xi }e = \lambda \varphi e , \tag{4.2}
\end{equation}
\begin{equation}
\nabla _{e}\varphi e = -\gamma e + \alpha \xi , \quad
\nabla _{\varphi e}\varphi e = \delta e-\beta \xi , \quad
\nabla _{\xi }\varphi e = - \lambda e , \tag{4.3}
\end{equation}
where $\gamma $, $\delta $, $\lambda $ are smooth functions defined on $U$.
Using the above equations, we compute%
\begin{equation*}
R(e,\varphi e)\xi =\left( e(\alpha )-\varphi e(\beta )\right) e+\left(
e(\beta )+\varphi e(\alpha )\right) \varphi e
\end{equation*}%
\begin{equation*}
R(\varphi e,\xi )e=\left( \varphi e(\lambda )+\xi (\delta )+\beta \delta
-\gamma \alpha -\gamma \lambda \right) \varphi e+\left( \xi (\alpha
)+2\alpha \beta \right) \xi
\end{equation*}%
\begin{equation*}
R(\xi ,e)\varphi e=\left( e(\lambda )-\xi (\gamma )-\beta \gamma -\delta
\alpha -\delta \lambda \right) e+\left( \xi (\alpha )+2\alpha \beta \right)
\xi \text{.}
\end{equation*}%
Adding these three equations, we conclude that%
\begin{equation}
e(\alpha )-\varphi e(\beta )+e(\lambda )-\xi (\gamma )=\beta \gamma +\delta
\alpha +\delta \lambda \text{,}  \tag{4.4}
\end{equation}%
\begin{equation}
e(\beta )+\varphi e(\alpha )+\varphi e(\lambda )+\xi (\delta )=\gamma \alpha
+\gamma \lambda -\beta \delta \text{,}  \tag{4.5}
\end{equation}%
and the third component gives the result in the Lemma~\ref{lem-2-1}. Also, we have

\begin{equation*}
R(\xi ,e)e=\left( \xi (\gamma )-e(\lambda )+\beta \gamma +\alpha \delta
+\lambda \delta \right) \varphi e+\left( -\xi (\beta )+\alpha ^{2}-\beta
^{2}\right) \xi
\end{equation*}%
and%
\begin{equation*}
R(\xi ,\varphi e)\varphi e=\left( \xi (\delta )+\varphi e(\lambda )+\beta
\delta -\alpha \gamma -\lambda \gamma \right) e+\left( -\xi (\beta )+\alpha
^{2}-\beta ^{2}\right) \xi
\end{equation*}%
Using the two equations above in $Q(\xi )=R(\xi ,e)e+R(\xi ,\varphi
e)\varphi e$, we obtain%
\begin{eqnarray*}
Q(\xi ) &=&\left( \xi (\delta )+\varphi e(\lambda )+\beta \delta -\alpha
\gamma -\lambda \gamma \right) e \\
&&+\left( \xi (\gamma )-e(\lambda )+\beta \gamma +\alpha \delta +\lambda
\delta \right) \varphi e \\
&&+2\left( -\xi (\beta )+\alpha ^{2}-\beta ^{2}\right) \xi \text{.}
\end{eqnarray*}%
This together with the equations (4.4) and (4.5) gives%
\begin{equation}
Q(\xi )=-\left( e(\beta )+\varphi e(\alpha )\right) e+\left( e(\alpha
)-\varphi e(\beta )\right) \varphi e+2\left( -\xi (\beta )+\alpha ^{2}-\beta
^{2}\right) \xi \text{.}  \tag{4.6}
\end{equation}%
Recall that in Theorem~\ref{th-3-2}, the vector field $\xi $ being an eigenvector of
the Ricci operator corresponding to a non-zero eigenvalue makes the
trans-Sasakian manifold homothetic to a Sasakian manifold. \ Similarly, we
have the following characterization of cosymplectic manifolds.

\begin{theorem} \label{th-4-1}
Let $(M,\varphi ,\xi ,\eta ,g,\alpha ,\beta )$ be a $3$%
-dimensional compact and connected trans-Sasakian manifold. Then $M$ is a
cosymplectic manifold if and only if the Ricci operator $Q$ annihilates the
vector field $\xi $.
\end{theorem}

\begin{proof}
Suppose that $Q(\xi ) = 0$ holds. Then (4.6) gives
\begin{equation}
e(\beta )=-\varphi e(\alpha )\text{,\quad }e(\alpha )=\varphi e(\beta )\text{%
, and }\xi (\beta )=\alpha ^{2}-\beta ^{2}\text{.}  \tag{4.7}
\end{equation}%
Applying Lemma~\ref{lem-2-1}
and the equations (2.2), (2.3), (4.1)-(4.3) and (4.7), we
obtain%
\begin{eqnarray*}
\Delta \alpha  &=&ee(\alpha )+\varphi e\varphi e(\alpha )+\xi \xi (\alpha
)-\nabla _{e}e(\alpha )-\nabla _{\varphi e}\varphi e(\alpha )-\nabla _{\xi
}\xi (\alpha ) \\
&=&[e,\varphi e](\beta )-2\xi (\alpha \beta )+\gamma e(\beta )-\delta
\varphi e(\beta )-4\alpha \beta ^{2} \\
&=&\left( \nabla _{e}\varphi e\right) (\beta )-\left( \nabla _{\varphi
e}e\right) (\beta )-2\xi (\alpha \beta )+\gamma e(\beta )-\delta \varphi
e(\beta )-4\alpha \beta ^{2} \\
&=&\left( -\gamma e+\alpha \xi \right) (\beta )-\left( -\delta \varphi
e-\alpha \xi \right) -2\xi (\alpha \beta )+\gamma e(\beta )-\delta \varphi
e(\beta )-4\alpha \beta ^{2} \\
&=&2\alpha \xi (\beta )-2\xi (\alpha \beta )-4\alpha \beta ^{2}=0.
\end{eqnarray*}%
Thus thanks to compactness of $M$ we have proved that $\alpha $ is a
constant. If $\alpha \neq 0$, then Lemma~\ref{lem-2-1}, implies that $\beta =0$ and
consequently the equation (4.7) gives $\alpha =0$, which is a contradiction.
Hence $\alpha =0$ and the equation (4.7) gives $\xi (\beta )=-\beta ^{2}$,
that is, $\text{div}(\beta \xi )=\beta ^{2}$, where we used $\text{div}\xi
=2\beta $, which follows from the equation (2.3). Using Stokes' theorem in $%
\text{div}(\beta \xi )=\beta ^{2}$, we obtain $\beta =0$. That is, $M$ is a
cosymplectic manifold. Conversely, if $M$ is a cosymplectic manifold, then
the equation (4.6) gives that $Q(\xi )=0$.
\end{proof}

\end{document}